\documentclass[12pt]{article}

\newlength{\stefan}
\usepackage{amssymb}
\usepackage{amsmath, amsthm, amsfonts, makeidx}

\usepackage[usenames,dvipsnames]{color}

\DeclareMathSymbol{\subsetneq}{\mathord}{AMSb}{"26}

\newtheorem{lemma}{Lemma}[section]

\newtheorem{theorem}[lemma]{Theorem}

\newtheorem{proposition}[lemma]{Proposition}
\newtheorem{corollary}[lemma]{Corollary}
\theoremstyle{definition}

\newtheorem{definition}[lemma]{Definition}
\newtheorem{example}[lemma]{Example}

\newtheorem{remark}[lemma]{Remark}

\newtheorem{question}[lemma]{Question}

\newcommand{\lp}{\longrightarrow}

\newcommand{\mb}{\mathbb}

\newcommand{\F}{\mb{F}}
\newcommand{\B}{\textup{B}}

\newcommand{\C}{\mb{C}}

\newcommand{\R}{\mb{R}}
\newcommand{\Z}{\mb{Z}}
\newcommand{\N}{\mb{N}}
\newcommand{\Q}{\mb{Q}}

\newcommand{\I}{\mathfrak{i}}
\newcommand{\J}{\mathfrak{j}}

\newcommand{\desda}{\Longleftrightarrow}

\renewcommand{\ker}{\operatorname{ker}}

\renewcommand{\deg}{\operatorname{deg}}

\newcommand{\im}{\textup{Im}}

\newcommand{\GL}{\operatorname{GL}}

\newcommand{\GA}{\operatorname{GA}}

\newcommand{\perm}{\operatorname{Perm}}
\newcommand{\BB}{\mathcal{B}}
\newcommand{\Int}{\operatorname{Int}}
\newcommand{\kar}{\operatorname{char}}
\newcommand{\Maps}{\operatorname{Maps}}
\newcommand{\INC}{\textsc{inc}}

\newcommand{\MA}{\operatorname{MA}}

\title{Triangular polynomial maps in characteristic $p$}
\author{
Stefan Maubach\\ \small Jacobs University Bremen\\\small 
Bremen, Germany \\ \small s.maubach@jacobs-university.de}

\begin{document}

\maketitle

\begin{abstract} 
This paper came to existence out of the desire to understand iterations of strictly triangular polynomial maps over finite fields. This resulted in two connected results:

First, we give a generalization of $\F_p$-actions on $\F_p^n$ and their description by ``locally iterative higher derivations'', namely $\Z$-actions on $\F_p^n$ and show how to describe them by what we call ``$\Z$-flows''. 
We prove equivalence between locally finite polynomial automorphisms (LFPEs) over finite fields and $\Z$-flows over finite fields. We elaborate on $\Z$-flows of strictly triangular polynomial maps.

Second, we describe how one can efficiently evaluate iterations of triangular polynomial permutations on $\F_p^n$ which have only one orbit. 
We do this by determining the equivalence classes in the triangular permutation group of such elements. We show how to conjugate them all to the map $z\lp z+1$ on the ring $\Z/p^n\Z$ (which is identified with $\F_{p^n}$), making iterations trivial.  An application in the form of fast-forward functions from cryptography is given.
\end{abstract}




\section{Preliminaries}

\subsection{Notations and definitions}

We will use the letter $k$ for a field. $p$ will be used exclusively for a prime integer. 
A polynomial endomorphism, or polynomial map, is an element $F=(F_1,\ldots,F_n)\in (k[X_1,\ldots,X_n])^n)$, which induces a map
$F:k^n\lp k^n$. Polynomial maps can be composed: $F\circ G=(F_1(G_1,\ldots, G_n),\ldots, F_n(G_1,\ldots, G_n))\in (k[X_1,\ldots,X_n])^n)$.  The polynomial map $I:=(X_1,\ldots, X_n)$ acts as identity with respect to composition, and given $n,k$ this defines a monoid, denoted $\MA_n(k)$. A polynomial map $F\in \MA_n(k)$ is invertible if there exists $G\in \MA_n(k)$ such that $F\circ G=I$; 
the set of invertible polynomial maps, or polynomial automorphisms, is denoted $\GA_n(k)$ (as a generalization of $\GL_n(k)$). 

\subsection{Introduction}

When generalizing the concept of algebraic $(k,+)$ actions on $k^n$ where $k$ is a field of characteristic zero, to fields of characteristic $p$, one tends to (obviously) go to $(k,+)$ actions on $k^n$. These then automatically have order $p$. This makes the generalization, though seemingly natural in some way, restrictive. For example, a common class of additive group actions is those induced by strictly triangular polynomial maps, that is, maps of the form $(X_1+g_1, \ldots, X_n+g_n)$, where $g_i\in k[X_1,\ldots, X_{i-1}]$. In characteristic zero all these maps can be embedded into a unique algebraic additive group action $\varphi: (k,+)\times k^n\lp k^n$ such that $\varphi(1,X_1,\ldots,X_n)$ is exactly this map. Analytically speaking, they are the ``time one-maps of a $(k,+)$ flow on $k^n$''. However, in characteristic $p$ they do not always have order $p$, so they cannot be part of a $(k,+)$-action.

To give an example, if $F=(X+Y+Z, Y+Z, Z)$ in characteristic zero, then the additive group action becomes 
\[ (t, (x,y,z))\lp (x+ty+\frac{1}{2}(t^2+t)z, y+tz, z).\]
Indeed, one can check that $(s,(t, (x,y,z))=(s+t,(x,y,z))$ and that $(1,(x,y,z))=F(x,y,z)$. If one considers the same linear map for the case that $\kar(k)=2$,  then the order of the map is 4, making it not part of a $(k,+)$-action on $k^3$. Another indication is that the above formula requires division by 2.

Let us elaborate on the $\kar(k)=0$ case, as a motivational background. For details we refer to \cite{FreudenburgBoek} and \cite{Essenboek}, we just present some of the basics very briefly.
A locally nilpotent derivation on a $k$-algebra $R$  is a $k$-linear map satisfying the Lebniz rule $D(fg)=fD(g)+D(f)g$ for all $f\in R$ (making it a derivation), and the property that for any $f\in R$ there exists $n\in \N$ such that $D^n(f)=0$ (making it locally nilpotent).
Given a locally nilpotent derivation $D$ on $k[X_1,\ldots,X_n]$ one can define a ring homomorphism
\[ \begin{array}{rl} \varphi_T: R\lp& R[T]\\
f\lp & \exp(TD)(f).
\end{array} \]
Since $D$ is locally nilpotent, the right hand side is a polynomial in $T$. 
Now one can prove that $\exp(SD)(\exp(TD)(f))=\exp((S+T)D)(f)$, and thus this ring homomorphism describes 
an additive group action $(t,x)\lp \exp(tD)(x)$. 
For example, if we take $D=(Y+\frac{1}{2}Z)\frac{\partial}{\partial X} + Z\frac{\partial}{\partial Y}$ on $k[X,Y,Z]$ then we get exactly the additive group action described above: $(t, (x,y,z))\lp (x+ty+\frac{1}{2}(t^2+t)z, y+tz, z)$. 
The other way around, given such an additive group action and corresponding ring homomorphism $\varphi_T: R\lp R[T]$, then one can define a locally nilpotent derivation by $f\lp (\frac{\partial}{\partial T} \varphi(f))\arrowvert_{T=0}$.

In characteristic zero it is not that hard to prove  that a strictly triangular polynomial map $F$ is the exponential of a unique locally nilpotent derivation: $F=\exp(D)$. (See prop. 2.1.13 in \cite{Essenboek}) Thus, such a map is uniquely connected with an additive group action and ring homomorphism $F_T: k[X_1,\ldots,X_n]\lp k[X_1,\ldots,X_n][T]$.  Let us denote $F_f=(F_T)\arrowvert_{T=f}$ for any $f$ for which this makes sense.\footnote{Meaning that $D(f)=0$, i.e. $f$ is in the kernel of $D$.} Then, in particular if $m\in \N$, $F_m=F^m=(F\circ F\circ\cdots \circ F)$. Hence, one has a polynomial parametrization of iterates of $F$, and given such a polynomial parametrization $F_T$ it becomes rather easy and efficient to compute iterates $F^m(v)$ where $v\in k^n$. 

In the case $\kar(k)=p$ the situation is widely different. First of all, one has the beforementioned discrepancy that not all strictly triangular polynomial maps are connected to additive group actions (not even {\em linear} strictly upper triangular maps are of this type). Secondly, the concept of locally nilpotent has to be replaced by ``locally iterative higer derivations'' which do not have many of the nice properties of locally nilpotent derivations (they are not determined by their result on the variables $X_1,\ldots,X_n$, for one). 

We solve this issue by providing a way to parametrize iterates in a different way, using integer-preserving polynomials (see section \ref{1var}) and in introducing $\Z$-flows (section \ref{vier}). This generalizes the concept of locally iterative higher derivations and is in our opinion more natural (it does not exclude certain unipotent linear maps, for one). 

Where, in general, providing a $\Z$-flow for a polynomial map $F$ gives a way to compute iterates  $F^m(v)$ somewhat efficiently, in  the case of finite fields we provide a different way to compute iterates of  $F^m(v)$. We restrict to the case where $F$ is a strictly upper triangular polynomial map of maximal order   (section \ref{tpm}). We achieve this (in part) by determining the equivalence class of such maps. In section \ref{vijf} we make a short computation on the computational weight of such iterations, and in section \ref{zes} we briefly indicate a link to fast-forward functions from cryptography.

\tableofcontents
\setcounter{tocdepth}{3}

\section{Triangular polynomial maps}

\label{tpm}

\subsection{The strictly  triangular permutation  group $\BB_n(\F_p)$}

Below, write $A_n:=\F_p[X_1,\ldots,X_n]$, and write $\I_n$ for the ideal in $A_n$ generated by the $X_i^p-X_i$. (We write $A$ if $n$ is clear.) Write $x_i:=X_i+\I_n$, and write $R_n:=\F_p[x_1,\ldots,x_n]=A_n/\I_n$.
Note that we can see $R_n$ as a subset of $R_{n+1}$ and identify $X_i+\I_n$ with $X_i+\I_{n+1}$, removing any ambiguity here. We will also write $\I$ in stead of $\I_n$ if $n$ is clear (or unimportant).

Note that the map $p: A_n\lp \Maps(\F_p^n,\F_p)$ satisfies $\I_n\subseteq \ker(p)$; but since $p$ is surjective, and
$\#(A_n/\I_n)=\#\Maps(\F_p^n,\F_p)$, we get that $\I_n=\ker(p)$. We thus can naturally identify $R_n$ with $\Maps(\F_p^n,\F_p^n)$. Thus, the natural map  $\pi:  (A_n)^n\lp \Maps(\F_p^n,\F_p^n)$ has kernel $(\I_n)^n$ (or $\I^n$), being a subset of $(A_n)^n$ (or $A^n$).

In this article, a polynomial map is an element $F\in (A_n)^n$. Each $F$ induces a map $\F_p^n\lp \F_p^n$,
and the above map $\pi$ is exactly the map assigning to each polynomial map its element in $\Maps(\F_p^n,\F_p^n)$. 
Hence, we may see $\pi(F)$ as an element of $(R_n)^n$, and since $\pi$ is surjective, these elements coincide one to one with the elements of 
$\Maps(\F_p^n,\F_p^n)$. So it means that we can write maps  like $(x_1^2+x_2, x_2+1+x_1)\in \Maps(\F_p^n, \F_p^n)$. 
The set of elements in $\Maps(\F_p^n,\F_p^n)$ which are isomorphisms we denote, as usual, by $\perm(\F_p^n)$. 

We define a polynomial map to be {\em triangular} if $F=(F_1,\ldots, F_n)$, where $F_i\in A_i=\F_p[X_1, X_{2},\ldots, X_i]$.\footnote{Note 
that often the definition is to let $F_i\in \F_p[X_i,\ldots, X_n]$ (and in fact we are used to it ourselves) but for this article it turned out 
to be more convenient to choose the definition in the text; in this case, some induction proofs have easier indexes). }
Similarly, $F$ is called {\em strictly triangular} if $F_i-X_i\in A_{i-1}=\F_p[X_1,\ldots,X_{i-1}]$. (Note that strictly triangular maps are automatically automorphisms; triangular maps in principle are not. In the latter case, invertibility is equivalent to 
$F_i-\lambda_i X_i\in
\F_p[X_1,\ldots,X_{i-1}]$ for some $\lambda_i\in\F_p$ for each $1\leq i\leq n$. )
We state that an element in $\Maps(\F_p^n,\F_p^n)$ is strictly triangular if it is the  image of a strictly triangular element in $A_n^n$.

The set of triangular polynomial automorphisms forms a subgroup (see \cite{FreudenburgBoek} section 3.6) denoted by $\B_n(\F_p)$, and the set of strictly triangular polynomial maps form a subgroup of $\B_n(\F_p)$, denoted by $\B_n^0(\F_p)$ (see \cite{hanoi} for the reasoning behind the naming of these groups). To be precise:
\[ B_n(\F_p)=\{(\lambda_1X_1+f_1,\ldots, \lambda_nX_n+f_n) ~|~ \lambda_i\in \F_p, f_i\in \F_p[X_1,\ldots, X_{i-1}]\}, \]
\[ B_0(\F_p)=\{(X_1+f_1,\ldots, X_n+f_n) ~|~  f_i\in \F_p[X_1,\ldots, X_{i-1}]\}. \]

We have $\pi(\GA_n(\F_p))\subseteq \perm(\F_p^n)$ (see \cite{SM03, MW09,  MW11} on the image of this group). Now one can also define the groups $\B_{n-m}(A_m), \B_{n-m}^0(A_m)$ and embed them naturally  into $
\B_n(\F_p), \B_n^0(\F_p)$.

\begin{definition} We will denote  $\pi(\B_n^{0}(\F_p))$ by $\BB_n(\F_p)$. 
We will call this group the strictly triangular permutation group.
\end{definition}

Elements $\sigma\in \BB_n(\F_p)$ thus have a unique representation of the form 
\[ \sigma=(x_1+g_1, x_2+g_2(x_1), \ldots, x_n+g_n(x_1,\ldots,x_{n-1})) \]
where we assume that $\deg_{x_i}(g_j)\leq p-1$ for each $1\leq i<j\leq n$. 
The group multiplication is the induced composition of polynomial automorphisms.
We will write $e=\pi(I)\in \perm(\F_p^n)$.

We define
\[ \BB_{n-m}(R_m):=\{ (x_1,x_2,\ldots, x_m, x_{1+m}+g_{1+m},\ldots, x_n+g_n) ~|~g_i\in \F_p[x_1,\ldots, x_{i-1}]\} \]
as the subgroup of $\BB_n(\F_p)$ fixing the first $m$ coordinates of $\F_p^n$.
Note that there exists a natural map $\B^0_{n-m}(A_m)\lp \BB_{n-m}(R_m)$ induced by the map $A_m\lp R_m$ and sending $(X_1+g_1, \ldots, X_{n-m}+g_{n-m})$ where $g_i\in \F_p[X_1,\ldots, X_{m+i-1}]$ to 
$(x_1,\ldots, x_{m}, x_{m+1}+g_1(x_1,\ldots,x_{m}), \ldots, x_n+g_{n-m}(x_1,\ldots, x_{n-1}))$.

We start with a few generalities/trivialities on elements of $\BB_n$. In particular, $\BB_n(\F_p)$ is a $p$-sylow subgroup of $\perm(\F_p^n)$.

\begin{lemma} \label{simpel}
Let $\sigma \in \BB_n(\F_p)$. Then
\begin{enumerate} \renewcommand{\labelenumi}{\roman{enumi}}
\item$ \BB_{n-m} (R_m)\lhd  \BB_n(\F_p) .$
\item$ \BB_{n-m}(R_m)/\BB_{n-m-k}(R_{m+k})\cong \BB_k(R_m). $ In particular, $\BB_{n-m}(R_m)/\BB_{n-m-1}(R_{m+1})\cong \BB_1(R_m)$, which is isomorphic with the group $<R_m,+>$.
\item $\BB_n(\F_p)\cong \BB_{m}(\F_p)\ltimes \BB_{n-m}(R_m)$.
\item  If $\sigma \in \BB_{n-m}(R_m)$, then $\sigma^p\in \BB_{n-m-1}(R_{m+1})$. 
\item If $\sigma \in \BB_n(\F_p)$, then $\sigma^{p^n}=e$.
\item Any cycle in $\sigma\in \BB_n(\F_p)$ has order $p^i$ for some $i$. 
\item $ \#\BB_{n-m} (R_m)=p^{
\frac{p^n-p^{m}}{p-1}
}.$
In particular, $\BB_n(\F_p)$ is a $p$-sylow subgroup of $\perm(\F_p^n)$. 

\item If $\gcd(m,p)=1$, then for $\sigma\in \BB_n(\F_p)$ there exists $\tau\in \BB_n(\F_p)$ such that $\tau^m=\sigma$. 
\end{enumerate}
\end{lemma}

\begin{proof}
{\em (i)} Let us consider the natural map $\BB_n(\F_p)\lp \BB_m(\F_p)$ (projection onto the first $m$ coordinates). 
Then it is easy to see that the kernel of this group homomorphism contains, and is contained in, $\BB_{n-m}(R_m)$, so the result follows.\\
{\em (ii)} A proof sketch to save space: modding out $\BB_{n-m-k}(R_{m+k})$ removes the last $n-m-k$ coordinates and leaves the first $m+k$ coordinates intact. To understand $\BB_1(R)$ for a ring $R$, note that elements are of the form $(x_1+r)$ and that $(x_1+r)\circ(x_1+s)=(x_1+r+s)$.\\
{\em (iii)} We have the natural stabilization embedding $\BB_{m}(\F_p)\lp \BB_n(\F_p)$, which is indeed a section for the 
quotient map $\BB_{n}(\F_p)\lp \BB_n(\F_p)/\BB_{n-m}(R_m)\cong \BB_m(\F_p)$, which proves the result.\\
{\em (iv)} Any element in $<R_m,+>$ has order $p$, hence if $\sigma \in \BB_{n-m}(R_m)$ then $\sigma+\BB_{n-m-1}(R_{m+1})\in \BB_{n-m}(R_m)/\BB_{n-m-1}(R_{m+1})$ has order $p$; hence $\sigma^p\in \BB_{n-m-1}(R_{m+1})$. \\
{\em (v)} Applying (iii) $n$ times, yields that if $\sigma \in \BB_n(\F_p)=\BB_n(R_0)$, then $\sigma^{p^n}\in \BB_{0}(R_n)$ which is the trivial group. \\
{\em (vi)} follows easily from (iv).\\
{\em (vii)}: The number of coefficients of $g_i$ is $p^{i-1}$, as it is a polynomial in the variables $x_1,\ldots, x_{i-1}$ bounded in each variable by degree $p-1$. Hence, an element in $\BB_{n-m}(R_m)$ is determined by  
$p^m+p^{m+1}+\ldots+p^{n-1}=p^{m}\frac{p^{n-m}-1}{p-1}$ coefficients. The stated formula follows since each coefficient can take $p$ values.\\
{\em (viii)} Since $(m,p^n)=1$ there exist $a,b\in \Z$ such that $am+bp^n=1$. Pick $\tau:=\sigma^a$, then $\tau^m=\sigma^{am}=\sigma$.
\end{proof}

\begin{remark}
In respect to lemma \ref{simpel} part (vii) we mention the papers of  Kaluznin from 1945 and 1947 \cite{kal1,kal2} which were motivated by finding the $p$-sylow subgroups of $\perm(N)$ where  $N\in \N$, $N\geq 1$. 
His description of the $p$-sylow groups of $\perm(p^n)$ is exactly the triangular permutation group. 
This example of a $p$-sylow group resembles the following well-known example: let $B:=\{I+N~|~N\in Mat_n(\F_p)$ strictly upper triangular$\}$ be the set of unipotent upper triangular matrices in $\GL_n(\F_p)$. Then $B$ is a $p$-sylow subgroup of $\GL_n(\F_p)$.
In fact, $B=\GA_n(\F_p)\cap \B_n(\F_p)$. 
\end{remark}

\subsection{Maximal orbit maps}

\begin{definition} We define $\sigma \in \BB_n(\F_p)$ being {\em of maximal orbit} if $\sigma$ consists of one permutation cycle of length $p^n$. 
\end{definition}

Next to the theoretical interest, our motivation for studying maximal orbit maps to the exclusion of arbitrary  maps, is for their possible applications in cryptography: if $\sigma \in \BB_n(\F_p)$ is  of maximal orbit which is unknown to an adversary, then choosing a random $m\in \Z$ and publishing $v\in \F_p^n, \sigma^m(v)\in \F_p^n$ gives absolutely no information on $\sigma$ nor $m$, as long as both stay hidden from the adversary (a broadcasted pair $(v,\sigma^m(v))$ is indistinguisheable from a random pair $(v,w)\in (\F_p^n)^2$). If $\sigma$ is not of maximal orbit, then $(v, \sigma^m(v))$ obviously does leak information on $\sigma$: it shows that $v$ and $\sigma^m(v)$ are in the same orbit (which excludes some maps).

The reason that we do not generalize the results of this section to other finite fields (i.e. finite extensions of $\F_p$) is that there exist no elements of maximal orbit in $\BB_n(\F_{p^m})$ if $m\geq 2$. (One can prove lemma \ref{simpel} part (i) for $\F_{p^m}$ for all $m$, so the longest possible orbit is $p^n$ instead of $p^{nm}$.)
We quote theorem 6 from \cite{Ostafe}:\footnote{The theorem in \cite{Ostafe} is more general, we restricted it to the case we need.}

\begin{theorem} \label{T3.2} $\sigma=(x_1+g_1,\ldots,x_n+g_n)$ is of maximal orbit if and only if the coefficient $c_i$ of $x_{1}^{p-1}\cdots x_{i-1}^{p-1}$ in $g_i$ is nonzero for each $1\leq i \leq n$.
Furthermore,  if $\sigma$ is of maximal orbit, then 
\[ \sigma^{p^{n-1}}(\tilde{\alpha},a)=(\tilde{\alpha},a+(-1)^{n-1}c_n ) \] for each $a\in \F_p,\tilde{\alpha}\in \F_p^{n-1}$.
\end{theorem}

So, the above theorem \ref{T3.2} gives a clear citerion in the coefficients appearing in $\sigma$  
for when an element in $\BB_n(\F_p)$ is of maximal orbit. 
Now, note that lemma \ref{simpel} part  (vii) actually tells one that it is possible to find an ``$m$-th root'' of any $\sigma\in \BB_n(\F_p)$ when $(m,p)=1$. For $m=p$, however, it will not be always possible. In particular, 
if $\sigma$ is of maximal orbit, it is not possible. This induces a few questions we were unable to solve satisfactorily like  theorem \ref{T3.2} does: 

\begin{question}~\\
(1) Can one recognise from the coefficients in $\sigma\in \BB_n(\F_p)$ if $\sigma$ is a $p$-th power of another map in $\BB_n(\F_p)$?
In particular, what is $\BB_{n-1}(R_1)/G$ where $G:=<\sigma^p ~|~ \sigma \in \BB_n(\F_p)>$?
\\
(2) Can one recognise from the coefficients in $\sigma \in \BB_n(\F_p)$ if $\sigma$ is a $p^i$-th power of a map of maximal orbit?\footnote{Added in proof: the author received a solution to this question in personal communication by Andreas Maurischat.}
\end{question}

Note that $G$ in (1) is a fully invariant subgroup of $\BB_n(\F_p)$, and in particular normal, see \cite{groepenboek} page 28.

There are some necessary requirements, like in (1) $\sigma$ must be in $\BB_{n-1}(R_1)$ and (consequently) in (2) $\sigma \in \BB_{n-i}(R_i)$, but these are by no means sufficient: $(x_1,x_2+x_1)$ is not a $p$-th power while $(x_1, x_2+1)$ is.

\subsection{Classification of maximal orbit maps}

We will consider the issue that if two maps are powers of each other, then they are interchangeable in some sense (as the set of iteratons of them is the same). After that we will find the conjugacy classes of maximal orbit maps.

\begin{definition} We say that two permutations $c, c'\in \perm(N)$, where $N\in \N^*$, are equivalent if 
$<c>=<c'>$, i.e. there exist $a,b\in \N^*$ such that $c^a=c', (c')^b=c$. 
\end{definition}

\begin{definition} \label{standardform} $\sigma=(x_1+g_1,\ldots, x_n+g_n) \in \BB_n(\F_p)$ is said to be in {\em standard form} if 
$\sigma (0,0,\ldots,0)=(1,0,0,\ldots,0,0)$, i.e. the constant terms of $g_2,\ldots, g_{n}$ are zero and $g_1=1$. 
\end{definition}


\begin{lemma} If $\sigma\in \BB_n(\F_p)$ is of maximal orbit, then there is exactly one $\sigma'\in \BB_n(\F_p)$ in standard form, such that
$\sigma,\sigma'$ are equivalent. In other words, standard form maximal orbit maps form a complete set of representatives of the maximal orbit maps
modulo equivalence.
\end{lemma}

\begin{proof}
Write $\sigma=(x_1+g_1,\tilde{\sigma})$. 
Since $\sigma$ is of maximal orbit, $g_1\not =0$. Now let $a\in \N$ be the inverse of $g_1$ modulo $p$. Then $\sigma^a=(x_1+ag_1,\ldots)=(x_1+1,\ldots)$ and   $\sigma^a$ is equivalent to $\sigma$. So we can assume that $g_1=1$ by replacing $\sigma$ by $\sigma^a$. 

Now, starting with $O:=(0,0,\ldots,0)$ and iterating $\sigma$, we see that $\sigma^m(O)=(m\mod p,\ldots)$. 
So, this first coordinate equals 1 if and only if $m\mod p=1$, which means that $m=ap+1$ for some $a\in \N$. 
Since $\sigma$ is of maximal orbit, the sequence $O, \sigma(O), \sigma^2(O), \ldots, \sigma^{p^n-1}(O)$ lists all elements of $\F_p^n$. 
The sublist of vectors starting with 1 is $\sigma(O), \sigma^{p+1}(O), \sigma^{2p+1}(O), \ldots, \sigma^{p^n-p+1}(O)$. 
One of these elements equals $(1,0,\ldots,0,0)$, i.e. there exists exactly one $a\in \N$ such that $\sigma^{ap+1}(O)=(1,0,\ldots,0,0)$. 
By lemma \ref{simpel} (vii) , $\sigma^{ap+1}$ is equivalent to $\sigma$, and satisfies the above requirement. Uniqueness is automatic, as for a cycle of length $p^n$ in $\perm(\F_p^n)$ there is only one power of that cycle sending $O$ to $(1,0,\ldots,0,0)$. 
\end{proof}

We will now focus on finding representatives for the conjugacy classes of maximal orbit maps. 

\begin{definition}
Write $x^{\alpha}=x_1^{\alpha_1}\cdots x_n^{\alpha_n}$ for $\alpha\in \F_p^n$. 
Define 
\[ R_n^{-}:=\sum_{\alpha\in \F_p^n, \alpha\not = (p-1,\ldots, p-1)} \F_p x^{\alpha} \]
the subvector space of $R_n$ without the monomial $(x_1\cdots x_n)^{p-1}$. \\
If $\sigma\in \BB_n(\F_p)$, define $\sigma^*: R_n\lp R_n$ by $\sigma^*(f)=f(\sigma)$. Note that $\sigma^*$ is a ring homomorphism and in particular additive.
We denote by $e^*$ the identity map on $R_n$. 
\end{definition} 

\begin{lemma} \label{kern}  $\sigma\in \BB_n(\F_p)$ is of maximal orbit if and only if $\ker(\sigma^* -e^*)=\F_p$. 
\end{lemma}

\begin{proof}
We will prove the following equivalences: Let $f\in \ker(\sigma^* -e^*)$. Then $0=\sigma^*(f)-e^*(f)=f(\sigma)-f$ so \\
\[ \begin{array}{rl}
& f\in \ker(\sigma^* -e^*)\\
\desda& f=f(\sigma)\\
\desda & f=f(\sigma^i) \forall i\in \N\\
\desda &  f(\alpha)=f(\sigma^i(\alpha)) \forall i\in \N, \alpha\in \F_p^n\\
\desda &  f\textup{~is~constant~on~orbits~of~}\sigma
\end{array} \]
If $\sigma$ has just one orbit, then $f$ is a constant function (and since $f\in R_n$ this indeed means $f\in \F_p$), and if $\sigma$ has at least two orbits, then 
$f$ need not be constant as it can obtain a different value for each orbit.
\end{proof}

\begin{corollary} \label{corrkern}
If $\sigma\in \BB_n(\F_p)$, then $\im(\sigma^* -e^*)\subseteq R_n^-$.
If $\sigma$ is of maximal orbit, then we  have equality $\im(\sigma^* -e^*)=R_n^-$.
\end{corollary}

\begin{proof}
Note that $\sigma^*(R_n^-)\subset R_n^-$. A computation shows that $(\sigma^*-e^*)((x_1\cdots x_n)^{p-1})\in \R_n^-$. 
Because of the linearity of $\sigma^*-e^*$ we thus have that $(\sigma^*-e^*)R_n=(\sigma^*-e^*)(\F_p (x_1\cdots x_n)^{p-1}+R_n^-)
\subseteq \F_p (\sigma^*-e^*)( (x_1\cdots x_n)^{p-1}) +(\sigma^*-e^*)(R_n^-)\subseteq R_n^-$. 

The second part follows from lemma \ref{kern}: the kernel has dimension 1, so the image must have codimension 1.
\end{proof}

\begin{proposition} \label{conjugate} Let $\sigma, \tau\in \BB_n(\F_p)$ of maximal orbit, i.e. 
\[ \sigma=(x_1+\lambda_1,~ x_2+\lambda_2x_1^{p-1}+g_2,~x_3+\lambda_3(x_1x_2)^{p-1}+g_3, \ldots,~ x_n+\lambda_n(x_1\cdots x_{n-1})^{p-1}+g_n),\]
\[ \tau=(x_1+\mu_1, x_2+\mu_2x_1^{p-1}+h_2,x_3+\mu_3(x_1x_2)^{p-1}+h_3, \ldots, x_n+\mu_n(x_1\cdots x_{n-1})^{p-1}+h_n),\]
where $\lambda_i, \mu_i\in \F_p^*$, and $g_{i},h_i\in R_{i-1}^-$. 
Then there exists $\varphi\in \BB_n(\F_p)$ such that $\varphi^{-1}\sigma\varphi=\tau$ if and only if $\lambda_i=\mu_i$ for all $1\leq i \leq n$. If $\varphi$ exists, then one may additionally assume $\varphi$ to be in standard form (see definition \ref{standardform}), and then $\varphi$ is unique.
\end{proposition}

The above proposition hence shows that $\lambda_1,\ldots, \lambda_n$ is a defining invariant for the conjugacy class of $\sigma$.

\begin{proof}
We use induction on $n$. The case $n=1$ is obvious (one picks $\varphi=(x_1+1)$, which is in standard form). Write $\sigma=(\tilde{\sigma}, x_n+g_n), \tau=(\tilde{\tau}, x_n+h_n)$ for some $g_n,h_n\in R_{n-1}$. 
The induction assumption means we can find a unique standard form map $\tilde{\varphi}$ in $n-1$ variables such that 
$\tilde{\varphi}^{-1}\sigma \tilde{\varphi}=\tilde{\tau}$ if and only if $\lambda_1=\mu_1,\ldots, \lambda_{n-1}=\mu_{n-1}$. 
We will extend $\varphi:=(\tilde{\varphi}, x_n)\phi$ where $\phi:=(x_1,\ldots, x_{n-1}, x_n+f_n)$ for some $f_n\in R_{n-1}$.
Write $(\tilde{\varphi}, x_n)^{-1} \sigma (\tilde{\varphi}, x_n)=(\tilde{\tau}, x_n+\lambda_n(x_1\cdots x_{n-1})^{p-1}+k_n)$
 where $k_n\in R_{n-1}^-$.
Now a computation reveals $\phi^{-1}(\tilde{\tau}, x_n+\lambda_n(x_1\cdots x_{n-1})^{p-1}+k_n)\phi= (\tilde{\tau}, x_n+\lambda_n(x_1\cdots x_{n-1})^{p-1}+k_n+(e^*-\tilde{\tau}^*)(f_n) )$.
We thus are (only) able to change $\lambda_n(x_1\cdots x_{n-1})^{p-1}+k_n$ by elements of $R_{n-1}^-$ as corollary \ref{corrkern} shows,
meaning that $\tau$ and $\sigma$ are only conjugate if $\lambda_n=\mu_n$. Let us assume the latter, and pick $f_n$ so that $(e^*-\tilde{\tau}^*)(f_n)=k_n$. If we assume $f_n$ to have constant part zero, then $f_n$ is unique. 
$\varphi$ is now on normal form by construction, and the above shows that it is unique.
\end{proof}

\begin{definition} Define $\delta_i\in R_i$ as the polynomial such that $\delta_i(p-1,\ldots, p-1)=1$ and $\delta_i(\alpha)=0$ for all
other $\alpha\in \F_p^i$. (And $\delta_0=1$.) Then define
\[\Delta:=(x_1+\delta_0,x_2+\delta_1,\ldots, x_n+\delta_{n-1}).\]
\end{definition}

\begin{theorem} \label{conjugate2}
Let $\sigma\in \BB_n(\F_p)$ be of maximal orbit. Then there exist a unique $\varphi\in \BB_n(\F_p)$ in standard form, and a diagonal linear map $D$, such that $D^{-1}\varphi^{-1}\sigma\varphi D=\Delta$. 
\end{theorem}

\begin{proof}
Write $\mu_i$ for the coefficient of $(x_1\cdots x_{i-1})^{p-1}$ in  $\delta_{i-1}$ ($\mu_1=1$).
By proposition \ref{conjugate} we see that $\sigma$ is equivalent to $(x_1+\lambda_1, x_2+\lambda_2\delta_1,\ldots, x_n+\lambda_n\delta_{n-1})$ for some $\lambda_i\in \F_p^*$.   Write $D:=(\lambda_1x_1,\ldots, \lambda_nx_n)$. 
By proposition \ref{conjugate} there exists a unique $\varphi\in \BB_n(\F_p)$ in standard form such that 
$\varphi^{-1}\sigma \varphi=(x_1+\lambda_1, x_2+\lambda_2\delta_1(D^{-1}), x_3+\lambda_3\delta_2(D^{-1}),\ldots,
x_n+\lambda_n\delta_{n-1}(D^{-1}))$. 
Now a computation reveals that $D^{-1}\varphi^{-1}\sigma \varphi D=\Delta$. 
\end{proof}

The above theorem thus enables us to see {\em all} maximal orbit maps as a unique conjugate of one map, namely $\Delta$. This map is, in some sense, very simple, as the following remark shows:

\begin{remark} \label{remark}
Define the bijection $\zeta: \Z/p^n\Z\lp  (\F_p)^n$ by $\zeta(a_0+a_1p+\ldots+a_{n-1}p^{n-1}) =(a_0,\ldots, a_{n-1}) \mod p $
where $0\leq a_i\leq p-1$. Then $\zeta^{-1}\Delta\zeta$ is the map $m\lp m+1$. 
\end{remark}

\section{Efficiently iterating maximal orbit triangular maps}
\label{vijf}

\subsection{Basic idea}

In some applications (the next section is an example) it might be necessary to efficently evaluate $\sigma^a(v)$ for a given $\sigma\in \BB_n(\F_p)$ of maximal orbit, and $a\in \Z$, $v\in \F_p^n$. Here we explain how to do this most efficiently, with respect to computation.

First, we find $\varphi$ and $D$ as given in theorem \ref{conjugate2}: thus $\sigma = D \varphi \Delta\varphi^{-1}D^{-1}$. 
Note that because of remark \ref{remark} it is trivial to compute $\Delta^a(v)$ for any given $v\in \F_p^n, a\in Z$: this part of the computation is negligible. We will consider any addition to be negligible anyway, and simply count the number of multiplications in $\F_p$
are needed. Hence, the evaluation $\sigma^a(v)$ needs 

\begin{itemize}
\item evaluations $D(v), D^{-1}(v)$,
\item evaluations $\varphi(v), \varphi^{-1}(v)$.
\end{itemize}

The storage of $\varphi$ does not immediately mean that $\varphi^{-1}$ is stored (or efficiently computable). However, the following
representation solves this: 

\begin{definition} Write $(x_i+g_i)$ for the map $(x_1,\ldots, x_{i-1}, x_i+g_i, x_{i+1}, \ldots, x_n)$. If
$g_i\in k[x_1,\ldots, x_{i-1},x_{i+1},\ldots, x_n]$ then its inverse is (as can be easily checked) $(x_i-g_i)$. 
\end{definition}

Note that if  $\varphi=(x_1+g_1,\ldots,x_n+g_n)$ then $\varphi= (x_1+g_1)(x_2+g_2)\cdots (x_n+g_n)$. 
Hence, $\varphi^{-1}=(x_n-g_n)(x_{n-1}-g_{n-1})\cdots (x_1-g_1)$. Thus, evaluation of $\varphi^{-1}(v)$ is of the same complexity as $\varphi(v)$, and it is not necessary to store anything extra.

\subsection{Storage size} \label{vijfpunttwee}

For completeness sake, we point out what the necesary storage size is.
With storage size we mean the amount of memory  required for storing a map $\sigma$ (or amount of memory necessary to fix which map $\sigma$ one has stored). 
The theoretical minimum storage size of a map $\sigma$ is bounded by the number of different elements in $\B_n(\F_p)$ of maximal orbit.
The amount of maximal orbit maps is approximately $p^{\frac{p^n-1}{p-1}}$ (i.e. one could store  $\frac{p^n-1}{p-1}$ coefficients in $\F_p$).

However, if we want to store the useful description above, then one stores $D$,  $\varphi$ and $\Delta$, which is approximately double of 
that,
i.e. we have to store approximately $2\frac{p^n-1}{p-1}$ coefficients in $\F_p$ in order to fix $\sigma$.

\subsection{Efficiency}

Since the generic audience of this journal might not be familiar with the topic of this section, we give a very brief background and mention some standard facts.
In this section we try to point out what is the  complexity (efficiency) of computations. As is custom, we neglect any additions in $\F_p$ as it is negligible with respect to the computational weight multiplication in $\F_p$ takes.  Hence, we need to determine how many multiplications are necessary. 

For example, if one has an arbitrary element $x\in \F_p$, and $0\leq m\leq p-1$ arbitrary, then the amount of multiplications necessary to compute $x^m$ is at most $2\log(m)$. \footnote{$\log$ in this section is $\log_2$} 
One simple algorithm is to iteratively square $x$ (i.e. compute $x^{(2^a)}$, at most $\log(m)$  multiplications) and then multiply those factors which one needs (at most $\log(m)$). 

The most important factor in computational efficiency is the {\em order} of computational complexity, which means that one does not compute the exact formula for amount of multiplications  given in the parameters (which here are  $p$ and $n$), 
but focuses on the heaviest factor if the parameters become larger and larger (``big O notation''). For example, 
$7p^2+p+4\sim O(p^2)$, or, is of order $p^2$.  Or, above, computing $x^m$ is of order $\log(m)$.\\

Note that the below basic lemma can probably be improved (see for example \cite{EMM}). 

\begin{lemma}\label{L5.5}
Let $f\in \F_p[x_1,\ldots,x_k]$ where $k\geq 1$ and $\deg_{x_i}(f)\leq (p-1)$ arbitrary. Then the expected amount of multiplications to evaluate $f$ is at most of order $p^{k-1}$. 
\end{lemma}

\begin{proof}
In any computation we need to evaluate  $x_i^{m}$ for each $m\leq p-1$, which costs in the order of $\log(m)$ multiplications, which will turn out to be negligible. We proceed by induction to $k$. If $k=1$ we need less than $1$ multiplication. Assume we have proven the theorem for $k-1$. Then $f\in \F_p[x_1,\ldots,x_n]$ means 
$f=\sum_{i=0}^{p-1} f_i x_k^i$ where $f_i\in \F_p[x_1,\ldots,x_{k-1}]$.
so we need to evaluate the $f_i$ and for all (but $f_0$) we need to multiply them by $x_k^i$. So we need order
 $p\cdot p^{k-2} =p^{k-1}$ multiplications.
\end{proof}

Note that a more detailed computation reveals that the number of multiplications necessary is actually very close to $p^{k-1}$ (not only of order $p^{k-1}$, which means it can be for example a large constant times $p^{k-1}$).

\begin{lemma} \label{eval} If $\varphi\in \BB_n(\F_p)$, then evaluation $\varphi(\lambda)$ for some $\lambda\in \F_p^n$ takes order
$\frac{p^{n-1}-1}{p-1}$ 
multiplications (or less than order $p^{n-1}$ multiplications).
\end{lemma}

\begin{proof}
If $\sigma=(x_1+g_1,\ldots, x_n+g_n)$ where $g_i\in R_{i-1}$, then evaluation of $\sigma$ means evaluationg the $g_i$. 
By lemma \ref{L5.5}, evaluation of $g_i$ ($i\geq 2$) costs order $p^{i-2}$ multiplications. Thus, we have possibly $1+p+p^2+\ldots+p^{n-2}=
\frac{p^{n-1}-1}{p-1}$ multiplications that have to be done. 
\end{proof}

Again, a more detailed computation reveals that the number of multiplications is actually very close to $\frac{p^{n-1}-1}{p-1}$. \\

{\bf Remark:} If $p=2$, then multiplication is of the same complexity as addition, so the author suspects that the above focus on ``amount of multiplications'' may be misleading. Nevertheless, we expect that the above computations are representative, and we expect especially the $p=2$ case to be very efficient and useful in applications.

\section{Fast-forward functions}
\label{zes}

Insipred by the previous section, we would like to briefly point out how  one can use the insights obtained in the previous section in a topic in cryptography.
According to Naor and Reingold \cite{NoarReingold}, a function $\sigma :\{0,\ldots, N-1\}\lp \{0,\ldots, N-1\}$  is fast forward if for each natural number $m$ which
is polynomial in $N$, and each $x\in \{0,\ldots, N-1\}$,  the computational
complexity of evaluating $\sigma^m(x)$ (the $m$-th iterate of $\sigma$ at $x$) is small, more precisely, is polynomial in $\log(N)$. This is useful in simulations and cryptographic
applications, and for the study of dynamic-theoretic properties of the
function $\sigma$. 

The previous section explains that any element in $\BB_n(\F_p)$ can be conjugated to the function $\INC:z\lp z+1$ on $\Z/p^n\Z$, making iteration trivial. This inspires us to construct fast-forward functions exactly the other way around: take the map 
$\INC$ and conjugate by some triangular polynomial map.
In this case, we consider permutations of $\F_p^n$, so $N=p^n$. If we pick generic triangular maps $\varphi$, then lemma \ref{eval} shows that we need order $p^{n}$ (let's say for simplicity: order $p^n$) computations. This means that computations are of order $N$, and not $\log(N)$, as desired by Noar and Reingold.

However, one can restrict to specific $\varphi$, which are less complicated, or at least less complicated to compute. Indeed, one can: we suggest the following implementation.

{\bf Suggested implementation}\\
We pick $p$ a prime  and $n$ a dimension. We pick $P\in \Z[x]$ a polynomial.
For $1\leq i \leq n$, Pick polynomials $f_i\in R_{i-1}$ which are {\em sparse}: they should be the sum of at most $P(n)$ monomials  $M\in R_{i-1}$,\footnote{Whereas a generic polynomial can have up to $p^n$ monomials.} i.e. $M=\lambda x_{i+1}^{a_{i+1}}\cdots x_n^{a_n}$ for some $\lambda, a_i\in \{0,\ldots,p-1\}$. 
Define $\tau_i:=(x_1,\ldots, x_{i-1}, x_i+f_i, x_{i+1},\ldots,x_n)$. 
(Note that $\tau_i^{-1}=(x_1,\ldots, x_{i-1}, x_i-f_i, x_{i+1},\ldots,x_n)$.)
For $1\leq i \leq n$, pick random $\lambda_i\in \F_p^*$ and define $D:=(\lambda_1x_1,\ldots, \lambda_nx_n)$. 
Then define 
\[ \sigma:=\tau_1\cdots \tau_n D\zeta \INC \zeta^{-1} D^{-1}\tau_n^{-1}\cdots \tau_1^{-1}. \]

{\bf Computational complexity for evaluationg $\sigma(\lambda)$.}\\
We can ignore the computational difficulty of any additions and the evaluations of $\INC$, $D$ and $\zeta$. 
Hence, the weight of computation is on the evaluation of the $2n\times P(n)$ monomials appearing in $\tau_1,\ldots, \tau_n, \tau_1^{-1},\ldots, \tau_n^{-1}$. Computing an evaluation of $x_i^{a_i}$ for some $a_i\in \{0,\ldots, p-1\}$ has order of complexity 
$\log(p)$ (i.e. order $\log(p)$ multiplications). Hence, evaluating a monomial $M_i$ costs at most order $n\log(p)+n$ (i.e. $n\log(p)$) multiplications. Then, evaluating $nP(n)$ such monomials yields $P(n)n^2\log(p)=nP(n)\log(p^n)$, and since $n=\log(p^n)/\log(p)$, computations are polynomial in  $\log(p^n)$. \\

Our goal is simply to drop here this idea of construction fast-forward functions in this way, but we can make a few remarks on which interested parties can continue.

\begin{itemize}
\item 
One of the motivations for studying fast-forward functions is for pseudo-random permutations, and in that respect one should prove cryptographic
computational {\em indistinguisheability} (see page 1 of \cite{NoarReingold}), which roughly means that an element we constructed ``looks'' the same as a random element. 
We expect that for  a map constructed as in the suggested implementation (with $P(n)$ at least of degree 1 in $n$), will be computationally indistinguisheable (in the cryptographic sense) from a random element in $\BB_n(\F_p)$ of maximal orbit.
\item 
If one is more interested in constructing elements in $\perm(\F_p^n)$ which have only one cycle (a desireable property for cryptographic applications), then one can conjugate a map $\sigma $ as constructed above by triangular maps of opposite orientation: i.e. construct $\mu_i$ exactly as one would construct the $\tau_i$, then $\mu:=(x_n,x_{n-1},\ldots,x_1)\mu_1\cdots \mu_n (x_n,x_{n-1},\ldots,x_1)$ is exactly such a ``lower triangular map''. Thus, construct
\[ \mu\tau_1\cdots \tau_n D\zeta \INC \zeta^{-1} D^{-1}\tau_n^{-1}\cdots \tau_1^{-1}\mu^{-1}.\]
We expect such maps to be computationally indistinguisheable from a random element in $\perm(\F_p^n)$ of maximal orbit, but a proof is probably difficult. Nevertheless, such a construction has the potential to be  a powerful cryptographic primitive.
\end{itemize}

{\bf Remark:} In \cite{Moh} a public key cryptosystem is given using polynomial automorphisms. See also the papers \cite{JH, GC} for cryptanalysis etc.


\section{Generalities on polynomial maps $\Z\lp \F_p$}
\label{1var}
The definitions below
first appeared in \cite{PO19}, but we present them here in the form introduced
in \cite{Ell11}.

\begin{definition} Let $A, B\subseteq \Q$. Then define  
\[ \Int(A,B):=\{f\in \Q[T] ~|~f(A)\subseteq B\}.\]
 In particular, we abbreviate $\Int(\Z)=\Int(\Z,\Z)$. Note that $\Int(A,B)$ is a subring of $\Q[T]$.
\end{definition}

The following is a well-known lemma: 

\begin{lemma} \label{L1} 
\[ \Int(\Z)=\bigoplus_{i\in \N} \Z{T\choose i}=\Z\left[{T\choose i} ~\vline ~i\in \N\right ]. \]
\end{lemma}

 A proof can be found in \cite{CC}, proposition I.1.1.

\begin{corollary} \label{L1a}\label{C2}
\[ \Int(\Z,\Z_{(p)})=\bigoplus_{i\in \N} \Z_{(p)}{T\choose i}=\Z_{(p)}\left[{T\choose i} ~\vline ~i\in \N\right].\]
\end{corollary}

\begin{proof}
$\supseteq$ is trivial, so let us assume $f\in \Int(\Z,\Z_{(p)})$. There exists some $d\in \Z\backslash p\Z$ such that $df\in \Int(\Z,\Z_{(p)})$ (since $f\in \Q[T]$ there  exist such $d\in\Z$, but since $f(\Z)$ does not involve denominators in $p\Z$, we can pick $d$ satisfying $(d,p)=1$). This means that $df\in \Int(\Z, \Z)= \Z\left[{T\choose i} ~\vline ~i\in \N\right ]$.
Hence, $f\in \Z_{(p)}\left[{T\choose i} ~\vline ~i\in \N\right]$.
\end{proof}

If $f\in \Z[{T\choose m}~|~m\in \N]$ then it makes sense to consider the map $\Z\lp \F_p$ given by $n\lp f(n)\mod p$. 
Also, if $r\in \Z_{(p)}$, then it makes sense to write down $r\mod p$ in the following way: 
if $r=\frac{a}{b}$ where $a\in \Z, b\in \Z\backslash p\Z$ then $r\mod p= (a\mod p)(b \mod p)^{-1}$. 

\begin{definition}\label{def2} Define $\tau: \Int(\Z, \Z_{(p)}) \lp \Maps(\Z, \F_p)$ by 
$\tau(f)(n)=f(n)\mod p$ for any $f\in \Int(\Z, \Z_{(p)})$. \\
We say that $f,g\in \Int(\Z, \Z_{(p)})$ are  {\em equivalent under $\tau$} if $\tau(f)=\tau(g)$. 
\end{definition}

\begin{remark}\label{RR} If $f\in \Int(\Z,\Z_{(p)})$ then there is some $g\in \Int(\Z)$ which is equivalent under $\tau$. 
\end{remark}

\begin{definition} Define $Q_i:={T\choose p^i}$. 
\end{definition}

\begin{proposition} \label{P3} Let $f\in \Int(\Z, \Z_{(p)})$ be of degree $d$. 
Then $f$ is equivalent to some $g\in \Z[Q_0,Q_1,\ldots, Q_r]$ where $r=[\log_p(d)]$. Furthermore, 
$g$ is at most of degree $p-1$ in each $Q_i$. 
\end{proposition}

The above proposition is based on Lucas' Theorem \cite{Lucas}:\\

\noindent
{\bf Lucas' Theorem:} Let $0\leq \alpha_i < p, 0\leq \beta_i <p$ where $\alpha_i,\beta_i\in \N$. Then 
\[ {\alpha_0+\alpha_1p+\alpha_2p^2+\ldots+\alpha_np^n \choose \beta_0+\beta_1p+\beta_2p^2+\ldots+\beta_np^n}
\equiv {\alpha_0 \choose \beta_0}{\alpha_1 \choose \beta_1}{\alpha_2 \choose \beta_2}\cdots {\alpha_n \choose \beta_n} \mod{p}.\]

\begin{corollary} \label{P4} Let $m\in \Z$, then  
$Q_i(m) = [ \frac{m}{p^{i}}] \mod{p}$, where $[{~}]$ is the floor (entier) function.
\end{corollary}

\begin{proof} Let $m=\alpha_0+\alpha_1p+\ldots+\alpha_np^n$. Then $[\frac{m}{p^i}]\mod{p}=\alpha_i$, 
and Lucas' theorem states that ${m\choose p^i}={\alpha_i\choose 1}=\alpha_i$. 
\end{proof}

\begin{proof}(of proposition \ref{P3}.) 
Let $f$ be as in the proposition. 
By corollary \ref{C2} $f$ is a $\Z_{(p)}$-linear combination of  ${T\choose 0}, {T\choose 1},\ldots, {T\choose d}$, 
which means by remark \ref{RR} that $f$ is equivalent to a $\Z$-linear combination of ${T\choose 0}, {T\choose 1},\ldots, {T\choose d}$. We will prove that $T\choose d$ is equivalent to an element of $\Z[Q_0,\ldots, Q_r]$ where $r=[\log_p(d)]$. 
Now, if $d=\alpha_0+\alpha_1p+\ldots+\alpha_rp^r$, and $m=\beta_0+\beta_1 p+\ldots + \beta_r p^r$ we use Lucas' Theorem again to derive the following: 
\[ \begin{array}{rl}
 {m\choose d}
&={\beta_0\choose \alpha_0}{\beta_1\choose \alpha_1}{ \beta_2\choose \alpha_2}\cdot\cdots \cdot { \beta_r\choose \alpha_r} \mod{p}.\\
\end{array} \]
Since $Q_i(m)=\beta_i$, we can state that the above is equal to 
\[ \begin{array}{rl}
\phantom{{m\choose d}}&={Q_0(m)\choose \alpha_0}{ Q_1(m)\choose \alpha_1}{Q_2(m)\choose \alpha_2}\cdot\cdots \cdot {Q_r(m)\choose \alpha_r} \mod{p}.
\end{array}\]
Now note that for $0\leq \alpha_i<p$, $\alpha_i\in \Z$, we can give 
\[ \frac{Q_i(T)(Q_i(T)-1)\ldots (Q_i(T)-\alpha_i+1)}{\alpha_i!}=:{Q_i(T)\choose \alpha_i}\in \Int(\Z, \Z_{(p)})\]
 and thus by the above considerations, we can state that ${T\choose d}\mod{p}$ coincides with
\[ {Q_0(T)\choose \alpha_0}{ Q_1(T)\choose \alpha_1}{Q_2(T)\choose \alpha_2}\cdot\cdots \cdot {Q_r(T)\choose \alpha_r} \mod{p}\]
for each substitution $T=m\in \Z$. This means exactly that the above functions are equivalent, proving the theorem. 
\end{proof}

\begin{corollary} \label{P5}
Let $f\in \Maps(\Z, \F_p)$ such that $f$ is periodic of order $p^r$. Then there exists $g\in \Z[Q_0,\ldots, Q_{r-1}]$ such that 
$\tau(g)=f$. 
\end{corollary}

\begin{proof} Proof sketch: using corollary \ref{P4}, we can take a linear combination of\\ $1,Q_i(T), Q_i(T)^2, \ldots, Q_i(T)^{p-1}$
to obtain a function which is periodic of period $p^i$, and is 1 on a particular interval $[ap^{i-1}, ap^{i-1}+(p-1)]$ where $a\in \{0,\ldots,p-1\}$ and 0 on other elements in $[0,p^i-1]\cap \Z$. Using these functions as a basis, the corollary follows easily. 
\end{proof} 

\section{$\Z$-flows }
\label{vier}

\subsection{$\Z$-flows of LFPEs}

Over a field $K$ of characteristic zero, given a strictly triangular polynomial map $F$, then it is always possible to give a formula for exponents $F^m$ of $F$, to be more precise:
 there is a strictly triangular polynomial map $F_T\in \GA_n(K[T])$ such that $F_m=F^m$ for each $m\in \N$. To give a simple (even linear) example: 

\begin{example} \label{Ex1} Let $F=(x+y+z, y+z, z)$, if $F_T:=(x+Ty+\frac{1}{2}(T^2+T)z, y+Tz, z)$, then $F_m=F^m$ for each $n\in \N$. 
\end{example}

However, if one picks $K$ a field of characteristic two, and considers the same map $F:=(x+y+z, y+z, z)$, then one runs into trouble
defining $F_T$, as it includes the polynomial $\frac{1}{2}(T^2+T)$. We will try to solve this problem, by introduction of so-called $\Z$-flows. To define this, we need some preparation.

\begin{definition} Given $k$ a field of characteristic $p$, define $B_n:=k[\tilde{Q}_0,\tilde{Q}_1,\ldots, \tilde{Q}_{n-1}]$ for $n\in \N$ where the $\tilde{Q}_i$ 
are independent variables, and $B:=\cup_{n\in \N} B_n$. Define $S_n:=B_n/\J_n$ where $\J_n:=(\tilde{Q}_i^p-\tilde{Q}_i|1\leq i\leq n)$ for all $n\in \N$, and $\J:=\cup_{n\in \N} \J_n$ and $S:=\cup_{n\in \N}S_n=B/\J$. We will write
$Q_i:=\tilde{Q}_i+\J$ or $Q_i:=\tilde{Q}_i+\J_n$. We can embed $S$ into $\Maps(\Z, k)$ by sending $Q_i$ to the map ${T\choose p^i}$. We will identify $S$ with its image in $\Maps(\Z,k)$ and identify  $Q_i$ with the elements ${T\choose p^i}$ in $\Maps(\Z, \F_p)$, all as described in section \ref{1var}.
\end{definition}

Indeed, notice that for each $m\in \Z$ we can define $Q_i(m)\in \F_p$, and since $k$ and $S_n$ are $\F_p$-modules we get a substitution map $\varphi_m:S_n\lp k$ 
\[ p(Q_0,\ldots, Q_{n-1})\lp p(Q_0(m),\ldots,Q_{n-1}(m)). \]
This map can be extended to a map 
\[ \varphi_m:\GA_n(S)\lp \GA_n(k). \]

\begin{definition} \label{flow} Let $F\in \GA_n(k)$ where $k$ is a field of characteristic $p$. 
Then we say $F_T\in \GA_n(S)$ is a $\Z$-flow of $F$ if $\varphi_m(F_T)=F^m$ for all $m\in \Z$.
\end{definition}

There is no real obstruction to extending the above definition to characteristic zero, but that would coincide with the already existing notion of ``locally nilpotent derivation'' and ``additive group action'', adding a third name to the list (and more confusion and obfuscation). In characteristic $p$ introduction of $\Z$-flows makes sense, as here the concepts of ``locally nilpotent derivation'', ``locally finite iterative higher derivation'' are truly different.  ``Additive group action'' normally means ``$(k,+)$-action'' whereas ``$\Z$-flow'' coincides with ``$(\Z,+)$ action''. 

Note that, getting back to example \ref{Ex1} in characteristic 2, we can define a $\Z$-flow 
$F_T:=(x+Q_0y+(Q_1+Q_0)z, y+Q_0z,z)$. 

 Even though definition \ref{flow} is for any field of characteristic $p$, we will now restrict to the case that $k$ is a finite field, and leave the generic case to a future paper.

\begin{remark}
(i) It is not true that all $F\in \GA_n(\F_p)$ are in the image of a $\Z$-flow,\footnote{Just as there does not always exist a locally nilpotent derivation $D$ such that $F=\exp(D)$ in the characteristic zero case.} as one needs $F$ to be locally finite (see below for the definition of this term), i.e. $\deg(F^m)$ is bounded as $m$ changes.\\
(ii) Locally iterative higher derivations correspond to $F_T\in \GA_n(\F_p[t])=\GA_n(S_1)$ such that $F^m=F_m$, i.e. a special subclass of $\Z$-flows (namely those having period $p$ - the additive $\F_p$ actions).\\
(iii) The wording $\Z$-flow comes from the analytic case: If $F$ is a holomorphic map $\C^n\lp \C^n$, then under some circumstances one can
define a holomorphic map $F_T: \C\times \C^n\lp \C^n$ such that $F_aF_b=F_{a+b}$ for each $a,b\in \C$, $F_1=F$ and $F_0=I$. 
Then $F_T$ is called a flow of $F$. \\
(iv) A $\Z$-flow is unique to $F$: if $F_T$ and $G_T$ are two $\Z$-flows, then $\varphi_m(F_T-G_T)=0$ for each $m$, which means that each coefficient is the zero function $\Z\lp k$, which corresponds to the zero element in $S$. 
\end{remark}

We will now explain for which polynomial maps there exist $\Z$-flows. 

\begin{definition} Let $F\in \MA_n(k)$ where $k$ is a field. Then we say that $F$ is a locally finite polynomial endomorphism (short LFPE) if $\{\deg(F^m)\}_{m\in \N}$ is bounded. \end{definition}

We quote results from theorem 1.1 and proposition 2.1 from \cite{Fu-Mau}: 

\begin{theorem}~\\ \label{LFPE1}
(i) $F\in\MA_n(k)$ is an LFPE if and only if there exist a linear dependence relation $F^d=a_{d-1}F^{d-1}+\ldots+a_1F+a_0 I$ where $a_i\in k$. Stating that $X^n-a_{d-1}X^{n-1}-\ldots -a_1X-a_0$ is a {\bf vanishing polynomial for $F$} in such a case, the set of vanishing polynomials for $F$ forms an ideal in $k[X]$.\\
(ii) $F\in \MA_n(k)$ is an LFPE which has a vanishing polynomial of the form $(X-1)^d$ for some $d\in \N$, if and only if $F=\exp(D)$ for some locally nilpotent derivation $D$ on $k[X_1,\ldots,X_n]$.
\end{theorem}

There is a conjectural equivalence between LFPEs and exponents of ``locally finite derivations''\footnote{Derivations which are locally finite linear maps on $k[X_1,\ldots,X_n]$. They include the locally nilpotent derivations, and their exponential maps exist if the field $k$ is closed under taking exponentials $x\lp e^x$, like $\R$ or $\C$.} for  $k=\C$. But for the special case of finite fields, we can give a very clear connection between LFPEs and $\Z$-flows:

\begin{proposition} 
Let $F\in \GA_n(\F_q)$ where $q=p^m$ for some $m\in \N^*$. Then for any $F\in \GA_n(k)$ we have equivalence between\\
(i) $F$ is an LFPE,\\
(ii) there exists $F_T\in \MA_n(S)$  which is a $\Z$-flow of $F$.
\end{proposition}

\begin{proof} 
If there exists a $\Z$-flow $F_T$ for $F$, then $F^m=\varphi_m(F_T)$ for each $m\in \Z$. This means that $\deg(F^m)$ is bounded by $\deg(F_T)$ and thus is $F$ an LFPE.

Now assume $F$ is an LFPE. 
Theorem \ref{LFPE1} (i) gives that  there exists some $d\in \N$ and a linear dependence relation $F^d=a_{d-1}F^{d-1}+\ldots+a_1F+a_0 I$ where $a_i\in \F_q$, i.e. $P(T):=T^d-a_{d-1}T^{d-1}-\ldots -a_0$ is a vanishing polynomial of $F$.
Since we are working over a finite field, there exists a polynomial $Q(T)$ such that $P(T)Q(T)=T^{q^r}-T$ for some $r\in \N^*$. By theorem \ref{LFPE1} (i) this means that $T^{q^r}-T$ is also a vanishing polynomial for $F$, i.e. $F^{q^r}=F$. Now using corollary \ref{P5} we find  $f_i\in S\subset \Maps(\Z, \F_q)$ having order $q^r$ sending $i$ to 1 and all other elements in the interval $[0, q^r-1]$ to 0. Now define $F_T:=\sum_{i=0}^{p^r-1} f_i F^i$. It is clear that $\varphi_m(F_T)=F^m$ for any $m\in \Z$ by construction.
\end{proof}

Over infinite fields, the connection between $\Z$-flows and LFPEs is not that all-encompassing: if $F=(\lambda x_1)\in \GA_1(k)$ where $\lambda\in k$ is no root of unity (i.e. $\lambda$ has no torsion), then $F$ is an LFPE which is not a $\Z$-flow, as there is no function in $S$ which equals $m\lp \lambda^m$. We expect that there is a deeper connection of $\Z$-flows similar to theorem \ref{LFPE1} (ii)  but that will require more research and will be addressed in a future paper.

We will also define $\Z$-flows in the following setting: Let $k=\F_q$ where $q=p^m$, $m\in \N^*$. Then we have the ring $S\subset \Maps(\Z, k)$. Given an element $F\in \MA_n(S)$, one can now consider the map induced by $F$, namely $\pi_q(F)\in \Maps(\Z\times \F_q^n,\F_q^n)$. We thus have a map 
\[ \pi_q: \MA_n(S) \lp \Maps(\Z\times \F_q^n,\F_q^n) .\]
Again we define $\varphi_m$ as the map substituting $m$ into the first component, i.e. 
\[ \varphi_m: \Maps(\Z\times \F_q^n,\F_q^n)\lp \Maps( \F_q^n,\F_q^n)\]

\begin{definition} 
Let $\sigma\in \pi_q(\GA_n(\F_q)$. Then we say that $\sigma_T\in \pi_q(\GA_n(S))$ is a $\Z$-flow of $\sigma$ if 
$\varphi_m(\sigma_T)=\sigma^m$ for all $m\in \N$. 
\end{definition}

\subsection{More general triangular groups}
\label{vierpunttwee}

If one has a ring $K$, then one can make the group $\textup{B}_n(K)$ and $\textup{B}^0_n(K)$ as described in section \ref{tpm}. But, it is possible to make slightly less intuitive groups: 
suppose that $K_1\subseteq K_2\subseteq \ldots \subseteq K_n\subseteq K$ is a chain of rings. Then one can make the set 
\[ \left\{(X_1+g_1,X_2+g_2,\ldots, X_n+g_n) ~\vline ~ g_i\in K_i[X_1,\ldots, X_{i-1}] \right\}\]
which becomes a subgroup of $\textup{B}^0_n(K)$. However, one can even make this work for more general subsets of $K$ which are not necessarily subrings.
First of all, if $f\in K[X_1,\ldots, X_n]$ and $F\in K[X_1,\ldots,X_n]^n$ then we can define $f\circ F=f(F_1,\ldots, F_n)$, which is needed below.

\begin{definition} Let $K$ be a ring and for $1\leq i \leq n$ let $W_i$ be a subgroup of $(K[X_1,\ldots, X_{i-1}],+)$ such that 
\[ W_i\circ (X_1+W_1,X_2+W_2, \ldots, X_i+W_i)\subseteq W_i .\]
Then define 
\[ \B(W_1,W_2,\ldots,W_n):=\left\{ (X_1+g_1,\ldots, X_n+g_n) ~\vline ~g_i\in W_i\right \} \]
which is a subset of $\B_n(K)$. 
\end{definition}

\begin{lemma}  $\B(W_1,W_2,\ldots,W_n)$ is a subgroup of $\B_n(K)$.
\end{lemma}

\begin{proof} The fact that the identity is in $ \B(W_1,\ldots,W_n)$ follows from the fact that $W_i$ is a subgroup and hence contains 0. 

We show that $\B(W_1,\ldots, W_n)$ is closed under composition: Let $G:=(X_1+g_1,\ldots, X_n+g_n), H:=(X_1+h_1,\ldots, X_n+h_n)\in \B(W_1,\ldots, W_n)$. Then $f_i:=g_i(H)\in W_i$ by assumption, and thus $G\circ H=(X_1+f_1,\ldots, X_n+f_n)\in \B(W_1,\ldots, W_n)$. 

We prove by induction that every element contains an inverse: $(X_1,\ldots, X_{i-1}, X_i+g_i, X_{i+1},\ldots, X_n)$ has an inverse $(X_1,\ldots, X_{i-1}, X_i+g_i, X_{i+1},\ldots, X_n)$. Now assume that all elements of the form $(X_1,\ldots, X_i, X_{i+1}+g_{i+1},\ldots, X_n+g_n)$ do have an inverse in $\B(W_1,\ldots, W_n)$. Pick $F:=(X_1,\ldots, X_{i-1},X_i+g_i, X_{i+1}+g_{i+1}, \ldots, X_n)$. Now $G:=(X_1,\ldots, X_{i-1}, X_i+g_i,X_{i+1},\ldots, X_n)\in \B(W_1,\ldots, W_n)$, and $FG^{-1}$ is invertible by induction. 
\end{proof}

Since one has a group homomorphism $\B_n^0(K)\lp \perm(K^n)$, there exists also a group homomorphism $\B(W_1,\ldots,W_n)\lp \perm(K^n)$.  We study the special case that $K$ is an $\F_p$-algebra  such that $r=r^p$ for each $r\in K$. 
(Given an $\F_p$-algebra, one can get such an algebra by modding out the kernel of the frobenius endomorphism $r\lp r^p$; one could also say that such an algebra is an $\F_p$ algebra with Frobenius automorphism being the identity.)

Now the map $\B^0 (S) \lp \perm(S^n)$ is a restriction of the map 
\[ \tau: S[X_1,\ldots,X_n]^n\lp S[x_1,\ldots,x_n]^n\lp \Maps(S^n,S^n)\]
 and thus it makes sense to write down $\BB_n(S)$, and we denote elements in this group like $\sigma:=(x_1+g_1,\ldots,x_n+g_n)$ where $g_i\in S[x_1,\ldots,x_n]$. Thus, we can also define the subgroup 
\[\BB(W_1,\ldots,W_n)\subset \BB_n(S)\]
where $W_i\subset S[x_1,\ldots,x_{i-1}]$. (Normally we should define this as $W_i\subseteq S[X_1, \ldots, X_{i-1}]$, but the groups coincide modulo $(X_1^p-X_1,\ldots,X_n^p-X_n)$ so this notation makes sense.)

In this article there are two such groups that we consider: remember that we defined  
$R_m:=\F_p[x_1,x_2,\ldots,x_m]$, $S_i:=\F_p[Q_0,\ldots,Q_{i-1}]/\j$ where $\j$ is generated by the $Q_j^p-Q_j$, $0\leq j\leq i-1$ and note that $S_iR_j=S_i\otimes R_j=S_i[x_1,\ldots,x_j]$. 
We will consider $\BB(S_1R_0, S_2R_1,\ldots,S_nR_{n-1})$ and the one mentioned in the next lemma.  

\begin{lemma} \label{vierpuntvier}
If $W_i:=S_{i-1}R_{i-1}+\F_pQ_{i-1}$, then $W_i\circ (x_1+W_1,\ldots,x_{i-1}+W_{i-1})\subseteq W_i$. 
Hence, $\BB(W_1,\ldots, W_n)$ is a subgroup of $\BB(S_1R_0,\ldots, S_nR_{n-1})$ and of $\BB_n(S_n)$. 
\end{lemma}

\begin{proof}
Let $g_i\in W_i$, i.e. $g_i=P(x_1,\ldots,x_{i-1})+\lambda Q_{i-1}$ where $P\in S_{i-1}R_{i-1}$. 
Let $h_j\in W_j$, then we need to prove that $P(x_1+h_1,\ldots,x_{i-1}+h_{i-1})+\lambda Q_i =g_i(x_1+h_1,\ldots,x_{i-1}+h_{i-1})\in W_i$. Now $x_j+h_j\in S_{j}R_{j-1}\subseteq S_{i-1}R_{i-1}$, and since $P\in S_{i-1}R_{i-1}$ we get 
$P(x_1+h_1,\ldots,x_{i-1}+h_{i-1})\in S_{i-1}R_{i-1}$ and we are done. 
\end{proof}

\subsection{$\Z$-flows of strictly triangular permutations}

\begin{theorem} \label{T4.4} Let $\sigma\in \BB_n(\F_p)$. Then\\
(1) there exists a $\Z$-flow $\sigma_T\in \BB(S_1R_0,S_2R_1,\ldots,S_nR_{n-1})$ of $\sigma$,\\
(2) $\sigma_T\in \BB(W_1,\ldots,W_n)$ where $W_i$ are as in lemma \ref{vierpuntvier}.
\end{theorem}

\begin{proof}
We use induction on $n$. For $n=1$, $\sigma=(x_1+a)$ where $a\in \F_p$, , and we can take $\sigma_T:=(x_1+Ta)\in x_1+R_0S_0+\F_pQ_0$. \\
Let $\sigma=(\tilde{\sigma}, x_n+g_n)\in \BB_n(\F_p)$. We know that we can find $\tilde{\sigma}_T\in \BB(W_1,\ldots,W_{n-1})$ such that $\sigma^m=(\tilde{\sigma}_m, x_n+h_m)$ where $h_m\in R_{n-1}$.
Now pick $H_m\in \Z[x_2,\ldots,x_n]$ such that $H_m\mod{p}=h_m$. 
Define 
\[ M_i(T):=\prod_{j=0, j\not = i }^{p^n-1} \frac{(T-j)}{i-j}  \]
and define $G(T):=M_0H_0+M_1H_1+\ldots+M_{p^n-1}H_{p^n-1}$. Note that $G(T)$ is of degree $p^n-1$ in $T$. Note that $G(i)=H_i$, and
$G(T)\in \Q[T][x_1,\ldots,x_n]$. Thus, if $c(T)$ is one of the coefficients in $\Q[T]$, then $c(\{0,1,\ldots, p^n-1\})\subset \Z$. 
Using lemma \ref{L1} we get that $c(\Z)\subset \Z$. 
Using proposition \ref{P3} we can replace each coefficient $c(T)\in \Q[T]$ by an equivalent element 
in $\Z[Q_0,Q_1,\ldots, Q_{n-1}]$ (as $[\log_p(p^n-1)]=n-1$), so we can assume that $G_T\in \Z[Q_0,\ldots, Q_{n-1}][x_1,\ldots,x_n]$.
Thus define $g_T\in \F_p[Q_0,\ldots,Q_{n-1}][x_1,\ldots,x_{n-1}]=S_{n}R_{n-1}$ as the image of $G_T$, and now we can define
\[ \sigma_T:=(\tilde{\sigma}_T, x_n+g_T) \]
and thus $\sigma_m=( \tilde{\sigma}_m, x_n+g_m)=( \tilde{\sigma}_m,x_n+h_m)=\sigma^m$, which is what is required.

Left to prove is that $g_T\in \F_pQ_{n-1}+S_{n-1}R_{n-1}$ (where we only have $g_T\in S_{n}R_{n-1}$ so far). Note that $\sigma^{p^{n-1}}(\tilde{\alpha},\alpha_n)=
(\tilde{\alpha}, \alpha_n+(-1)^{n-1}a)$ where $a$ is the coefficient of $(x_1\cdots x_{n-1})^{p-1}$ in $x_n+g_n$ (see theorem \ref{T3.2}). 
This means that $\sigma^m=(x_1+(-1)^{n-1}a, x_2,\ldots,x_n)$ if 
$p^{n-1}$ divides $m$. Write $\lambda=a(-1)^{n-1}\in \F_p$, then 
$g_{mp^{n-1}}=m\lambda$. 
Now define $h_T:=g_T-Q_{n-1}(T)\lambda$. Then $h_{p^{n-1}}=0$, and thus 
$h_T$ does not depend on $Q_{n-1}$ (which has order $p^n$). Thus, 
$h_T\in S_{n-1}R_{n-1}$ and $g_T\in S_{n-1}R_{n-1}+\F_p Q_{n-1}=W_n$. 
\end{proof}

It might be that this theorem can be improved, in the sense that the $W_i$ can be chosen smaller. This comes down to the following question:

\begin{question} \label{vraagvierpuntacht} Find $W_1,\ldots,W_n$ such that
\[ \BB(W_1,W_2,\ldots, W_n)= \left<\sigma_T ~|~\sigma \in \BB_n(\F_p) \right>.\]
\end{question}

We denote $t:=Q_0$, thus $\F_p[t]:=\F_p[T]/(T^p-T)$.

\begin{theorem}\label{Tt1}
Let $\sigma\in \BB_n(\F_p)$. Then there exist 
\[ \sigma_{i,T}\in
\BB(\F_p t, R_{i+1}[t], R_{i+2}[t],\ldots,R_{n-1}[t]) \subset 
 \BB_n(\F_p[t])\] for $0\leq i\leq n-1$ such that
$\sigma^{p^i m}=\varphi_m(\sigma_{i,T})$ for each $0\leq m\leq p-1$. 
\end{theorem}

\begin{proof}
Lemma \ref{Lt1} gives the case $i=0$. Defining $\tau:=\sigma^{p^i}$, then $\tau\in \B_{n-i}(\F_p)$, so we can apply lemma \ref{Lt1} to $\tau$ to find $\tau_{0,T}$; now define $\sigma_{i,T}:=\tau_{0,T}$, and $\sigma^{p^i m}=\tau^m=\varphi_m(\tau_{0,T})=\varphi_m(\tau_{i,T})$ for each $0\leq m\leq p-1$.
\end{proof}

\begin{lemma} \label{Lt1}Let $\sigma\in \BB_{n-i}(R_i)$. Then there exists 
\[ \sigma_{i,T}\in
\BB(\F_p t, R_{i+1}[t], R_{i+2}[t],\ldots,R_{n-1}[t]) \]
such that $\sigma^{m}=\sigma_{i,m}$ for each $0\leq m\leq p-1$. 
\end{lemma}

\begin{proof}
Let $M_i(t):=\prod_{j=0, j\not = i}^{p-1} \frac{t-j}{i-j}$.
Then define $\sigma_{0,T}=\sum_{i=0}^{p-1}M_i f^i$. It is now clear that $\sigma_{0,T}\in \BB_n(\F_p[t])$, one only needs to see that 
the first component is of the form $x_1+t\lambda$ for some $\lambda \in \F_p$. 
But since the first component of $\sigma$ is $x_1+\lambda$ for some $\lambda$, and thus $\sigma^m$ has $x_1+m\lambda$ as first component, this is exactly the case. 
\end{proof}


\end{document}